\documentclass[11pt]{amsart}
\usepackage{amsmath,amsfonts,amsthm,amscd,amssymb,graphicx}

\newtheorem{theorem}{Theorem}[section]
\newtheorem{lemma}[theorem]{Lemma}

\newtheorem{corollary}[theorem]{Corollary}

\newtheorem{remark}[theorem]{Remark}
\newtheorem{definition}[theorem]{Definition}

\numberwithin{equation}{section}

\def\R{\Re e}

\newcommand{\Span}{\text{\rm Span\ }}

\begin{document}
\title[Efficient numerical stability analysis]{Efficient
numerical stability analysis of detonation waves in ZND}
\author{Jeffrey Humpherys and Kevin Zumbrun}

\date{\today}

\thanks{ This work was supported in part by the National Science Foundation award numbers DMS-0607721 and DMS-0300487, and National Science Fountation CAREER award DMS-0847074. Thanks to Mark Williams for stimulating discussions regarding
the numerical literature on stability of ZND detonations.}

\address{Department of Mathematics, Brigham Young University, Provo, UT 84602}
\email{jeffh@math.byu.edu}
\address{Department of Mathematics, Indiana University, Bloomington, IN 47402}
\email{kzumbrun@indiana.edu}

\begin{abstract}
As described 
in the classic works of
Lee--Stewart and Short--Stewart,
the numerical evaluation of linear stability of planar detonation
waves is a computationally
intensive problem of considerable interest in applications.
Reexamining this problem from a
modern numerical Evans function point of view,
we derive a new algorithm for their stability analysis,
related to 
a much older method of Erpenbeck,
that, while equally simple and easy to implement as the
standard method introduced by Lee--Stewart, 
appears to be potentially faster and more stable.
\end{abstract}

\maketitle


\section{Introduction}

As described for example in \cite{Er2,Er3,FD,LS,ShS,SK},
the numerical stability analysis of detonation wave solutions
of the Zeldovich--von Neumann--D\"oring (ZND), or reactive Euler equations,
is a rich and computationally challenging problem. 
Planar detonation waves can often change stability as physical parameters
are varied, undergoing interesting bifurcations
to pulsating, spinning, and cellular solutions
\cite{BMR,FW,AT,KaS,LZ1,JLW,TZ3,TZ4,TZ5}.
This motivates the numerical study of their stability,
originated by Erpenbeck in \cite{Er2,Er3}, both for its
interest in its own right and as a benchmark for more general 
time-evolution codes \cite{BMR,SK}.

Due both to the number of physical parameters
(four for a polytropic gas\footnote{
Gas constant $\Gamma=\gamma-1$, heat release coefficient $q$, 
activation energy $E_A$, and detonation amplitude \cite{Er2,LS,Z6}.})
and the difficulty of individual computations,
this problem has proven to be numerically intensive.
In their classical 1990 paper \cite{LS}, in which they
introduced the algorithm that has become the modern-day standard,
computing accurately for the first time
the stability boundaries for one-dimensional detonations, 
Lee and Stewart conclude (p. 131 of the reference):
``Finally, we point out that though our scheme is direct and easy 
to implement, complete investigation of the various regions of
parameter space is computationally intensive.
Any equivalent or more efficient numerical method should be considered
a valuable contribution and such approaches are needed to further 
explore the parameter regimes of instability.''

Despite these comments, the basic algorithm introduced by Lee-Stewart (or perhaps variants thereof) as described in the 2006 survey \cite{SK}  appears still to be the current state of the art.
Of course, computational power has increased tremendously in the interim, making once-prohibitive computations now accessible.  Nonetheless, it seems of interest to explore more efficient algorithms if they can be found.

In particular, the computations of \cite{LS} were carried out 
in 1990 on a Cray X-MP/48 supercomputer,\footnote{A Cray X-MP/48 cost roughly \$15-20M dollars in the mid-1980's, having 2 processors with a 105 MHz clock speed and a theoretical peak performance about 200 MFLOPS per processor or 400 MFLOPS total.} with several hours required to produce individual figures.
(For example, Fig. 9 of \cite{LS} tracking the top $6$ unstable eigenvalues of detonations of a polytropic gas with gas constant $\gamma=1.2$ as activation energy is varied was reported to require $5$ hours of computation.)
Today, substantially more computing power is available in a standard desktop PC, and a relatively inexpensive multi-core workstation offers substantially more.\footnote{A 2010 Mac Pro $8$-core (2 quad-core Xeon processors) for example is a \$4-5k system with a 2.5GHz clock speed and a theoretical peak performance around 10 GFLOPS per core or 80 GFLOPS total. Hence, it has roughly $200$ times the processing power at a five thousandth the price (not even adjusting for inflation).}
Hence, the challenge is transposed from the level of the national lab to
the level of individual users, and from feasibility to practical ease of use.
However, the impetus is no less real to reduce computation time from
hours to the minutes required for interactive numerical explorations,
and such improvement would undoubtedly lead to further advances in our
understanding of detonation phenomena.

Meanwhile, in parallel development, there has been considerable 
activity, centered around the {\it Evans function} \cite{AGJ,PW,GZ}, 
in the numerical evaluation of stability of viscous shock
waves and other traveling front or pulse and boundary layer solutions 
arising in a variety of equations
\cite{Br1,Br2,BrZ,BDG,HuZ2,BHRZ,HLZ,HLyZ,BHZ,BLZ,BLeZ,CHNZ,BSYZ},
some of which problems- see, e.g., \cite{HLyZ,BHZ,BLZ,CHNZ} exhibit
complexity rivalling that of detonations.
The authors and collaborators have developed a general model-independent  method and set of numerical principles for the treatment of such problems \cite{HuZ2,Z5}, encoded in the MATLAB-based platform STABLAB \cite{STABLAB}, which performs extremely well on all of the above-described applications.

At the same time, there has been a successful push to place
detonation stability in a common 
framework with stability of shock waves \cite{Z1,LZ1,LZ2,JLW,TZ5,Z3,Z6}.
In particular, in \cite{Z1,JLW,TZ5,Z3,Z6}, 
the determination of stability of both
viscous (reactive Navier--Stokes) and inviscid (reactive Euler or ZND) 
detonations has been reduced to the computation of an Evans function 
defined exactly as in the viscous shock and other cases described above.
Thus, it is a natural step to study ZND stability within this common
framework, using the general tools of \cite{HuZ2,Z5}.

In this paper, we do exactly that, proposing a new algorithm for
the numerical determination of stability of ZND detonations derived
from the point of view of \cite{HuZ2,Z5}.
{\it Surprisingly, though both are shooting methods,
this is quite different from the Lee-Stewart algorithm
currently in standard use}, shooting from $x=-\infty$
to $x=0$ rather than from $x=0$ to $x=-\infty$ as in \cite{LS};
indeed, it is more closely related to the original algorithm
of Erpenbeck \cite{Er3}.
The precise relations between the various
methods are described in Section \ref{s:relations}.

The advantage of shooting from $-\infty$ to $0$ is that we seek
generalized eigenfunctions decaying exponentially at $-\infty$.
Thus, in the forward direction ($-\infty \to 0$), the desired
solution grows exponentially, while error modes are exponentially
damped.  By contrast, integrating in the backward direction
($0\to -\infty$), the desired solution decays exponentially while
error modes are exponentially {\it amplified}, a numerically undesirable
situation
(``numerical pitfall 1'' of \cite{Z5}).
For this reason, we expect that our algorithm should be faster and
better conditioned than the Lee-Stewart algorithm currently in use.
However, there are other aspects that cloud the issue, in particular
the singular perturbation structure that arises in the
high-activation energy or ``square-wave'' limit in which instabilities
are often studied \cite{Er2,FD,FW,BN,AT}.
For this reason, careful comparison of methods in physically relevant
regimes is an important step before making conclusions.

In the present paper, we introduce the algorithm, and give some
supporting numerical experiments for a simple model equation 
indicating the advantages of our approach.
Followup work in \cite{BZ,BZ2} indicates that, also in
physically realistic settings, the algorithm performs favorably
compared to the current standard.  Specifically, the standard
adaptive-mesh version of the algorithm described here appears
to outperform the fixed-mesh algorithm described in \cite{LS,SK} 
by $2$-$3$ orders of magnitude.
Much of this improvement appears to be due to the difference between
fixed and adaptive mesh. 
However, even compared to an adaptive-mesh version
of the method of Lee-Stewart, our algorithm appears to be 
$1$-$10$ times faster, depending on the parameter regime: at the
least, it is equivalent, and in some situations substantially more
efficient.

\medskip

{\bf Plan of the paper.}
In Section \ref{s:ZND}, we review the ZND equations and detonation
structure.
In Section \ref{s:linearized}, we give a simple derivation
of the Evans/Lopatinski function condition for detonation stability from
a general point of view following \cite{Z1,JLW}.
For clarity, we specialize in most of the discussion to the
single-species, ideal gas case with Arrhenius ignition dynamics, 
working in the same framework as in \cite{LS}.
The general case is discussed briefly in Remark \ref{general}.
In Section \ref{s:relations}, we determine the relation between
the derived Evans/Lopatinski condition the related
stability determinants of Erpenbeck \cite{Er2} and Lee-Stewart \cite{LS}.
In Section \ref{s:algorithm}, we describe a proposed numerical
implementation within the standard STABLAB package developed
by the authors and collaborators.
Finally, in Section \ref{s:simplemodel},
we present numerical experiments for a simple model indicating the
advantages of integrating in the forward direction and factoring
out expected decay at $-\infty$ as prescribed in \cite{HuZ2,Z5}.


%

\section{ZND detonations}\label{s:ZND}

\subsection{The model}\label{model}
In Eulerian coordinates the Zeldovich--von Neumann--D\"oring (ZND)
equations of reacting gas dynamics in one space dimension
may be written as
\begin{equation}\label{znd}
\begin{aligned}
\rho_t + (\rho u)_x & = 0 \\
(\rho u)_t + (\rho u^2+p)_x & = 0 \\
(\rho \tilde{E})_t+((\rho \tilde{E}+p)u)_x & = 0\\
(\rho Y)_t+(\rho uY)_x & =  -\rho \varphi(T) K Y, \\
\end{aligned}
\end{equation}
where $\rho$, $u$, $p$, $\tilde E$, $T\in \mathbb{R}^1$ 
represent density, 
velocity, pressure, total energy, and temperature, and 
$Y=(Y_1, \dots, Y_r)\in \mathbb{R}^r$ the mass fractions of 
reactants.\footnote{ Alternatively, the equations may 
be written in terms of progress variables $\lambda_j=1-Y_j$ \cite{FD,LS,LZ2}.} 
Here, $\tilde E=u^2/2+\tilde e$ is the
non-reacting gas-dynamical energy 
$E=u^2/2+e$ modified by chemical potential according to
\[
\tilde e= e+qY,
\]
where $e$ is the specific internal energy of the gas and $qY$ is the
specific chemical energy.
The matrix $K\in \mathbb{R}^{r\times r}$ and vector $q\in \mathbb{R}^{1\times r}$ 
measure the rates of reaction and the heat 
released in reaction, respectively, and $\varphi$ is an ``ignition function''
that is positive for $T$ above some ignition temperature $T_i$ and
zero for $T\le T_i$, serving to ``turn on'' the reaction.
The matrix $-K$ is assumed to be stable, i.e., to have spectrum
of strictly negative real part, so that reaction in a quiescent
flow indeed proceeds to the completely burned state $Y=0$.
In the simplest case of a single-species, exothermic reaction,
$Y\in \mathbb{R}^1$ is a scalar, and $K$ and $q$ are positive constants.

The system is closed by specifying equations of state (i.e.,
thermodynamic relations) $p=p(\rho , e, Y)$ and $T=T(\rho,e,Y)$
and the ignition function.  
Standard assumptions (in particular, the ones made in \cite{LS}, etc.)
are the ideal gas laws
\begin{equation}
\label{ideal}
p(\rho,e)=\Gamma\rho e,
\quad T(e)= e/C_v,
\end{equation}
where $\Gamma$, $C_v>0$ are constants determined by the nature of the gas,
and the modified Arrhenius law
\begin{equation}\label{Arrhenius}
\varphi(T)=\exp\left(-\frac{E_A}{RT}\right)\beta(T),
\end{equation}
where $E_A$ is the activation energy, $R =\gamma C_V$ is the gas constant,
and $\beta$ is an artificial smooth cutoff function with
the property that $\beta\equiv 1$ for $T\ge T^i$ and $\beta\equiv 0$
for $T\le T_i$.\footnote{The latter, standard modification 
circumvents the ``cold-boundary difficulty'' that the unburned state
$Y\equiv 1$ is not an equilibrium for the exact Arrhenius law $\beta\equiv 1$,
and so steady traveling detonation waves do not exist.
Though not mentioned, this assumption is also made implicitly 
in \cite{LS}, etc.}
Under usual assumptions, the specific form of the function $\beta$
plays no role in the analysis; see Remark \ref{beta}.


\begin{remark}\label{rate}
More realistic rate laws $r(\rho, T, Y)$ may be
considered in place of the linear law $r=-\rho \varphi(T)K Y$
with little additional difficulty \cite{LS}; however,
we lose the explicit form of the reaction profile \eqref{Yform}
computed in Section \ref{init}.
In the single-species case, these are equivalent.
\end{remark}
\subsection{Alternative formulation}\label{form}

Subtracting $q$ times the fourth equation of \eqref{znd}
from the third equation, we obtain the alternative formulation
\begin{equation}\label{altznd}
\begin{aligned}
\rho_t + (\rho u)_x & = 0 \\
(\rho u)_t + (\rho u^2+p)_x & = 0 \\
(\rho {E})_t+((\rho {E}+p)u)_x & = \rho q \varphi(T) K Y \\
(\rho Y)_t+(\rho uY)_x & =  -\rho \varphi(T) K Y \\
\end{aligned}
\end{equation}
in terms of the usual gas-dynamical variables $\rho$, $u$, $E$.
We alternate between the two formulations as convenient for the analysis.

\subsection{Detonation waves}\label{detonation}

For temperatures $T\le T_i$ below igition level, equations
\eqref{znd}
evidently reduces to the usual 
Euler equations of nonreactive gas dynamics, with the reactants $Y$ 
convected passively by the velocity field $u$.
In particular, so long as $T(\rho_\pm, e_\pm, Y_0)\le T_i$,
they support as traveling-wave solutions ordinary
gas-dynamical shock waves
\[
(\rho, u, E, Y)(x-st)=
\begin{cases}
(\rho_+, u_+, E_+, Y_0) & x-st >0\\
(\rho_-, u_-, E_-, Y_0) & x-st \le 0\\
\end{cases}
\]
satisfying the Rankine--Hugoniot conditions
\begin{equation}\label{RHprelim}
s[\rho]=[\rho u], \quad
s[\rho u]=[\rho u^2+p],\quad
s[\rho E]= [(\rho {E}+p)u],
\quad [Y]=0,
\end{equation}
or, equivalently,
\begin{equation}\label{RH}
s[\rho]=[\rho u], \quad
s[\rho u]=[\rho u^2+p],\quad
s[\rho \tilde E]= [(\rho {\tilde E}+p)u],
\quad [Y]=0,
\end{equation}
where for an arbitrary function $h(\rho,u,\tilde E,Y)$, 
$[h]:=h(\rho_+, u_+,\tilde E_+, Y_+) - h(\rho_-,u_-, \tilde E_-, Y_-)$ 
denotes jump across the discontinuity.  This also holds if there is no reactant, $Y_0=(0,\dots,0)$.

If, on the other hand, $Y_0 \ne (0,\dots,0)$, and $T_+\le T_i$
but $T_-\ge T_i$, with $u_\pm <s$ 
(alternatively, $T_+\ge T_i$ and $T_-\le T_i$, with $u_\pm>s$),
then there appears a different type of traveling-wave solution
known as a {\it strong detonation}, given by ($z=x-st$)
\begin{equation}\label{deteq}
(\rho, u, E, Y)(z)=
\begin{cases}
(\rho_+, u_+, E_+, Y_0) & z >0\\
(\bar \rho, \bar u, \bar E, \bar Y)(z) & z \le 0,\\
\end{cases}
\end{equation}
where $\bar Y(z)$ satisfies the smooth traveling-profile ODE
\begin{equation}\label{preY}
(\bar \rho (\bar u-s) \bar Y)'= -\bar \rho \varphi(\bar T) K \bar Y 
\end{equation}
on $(-\infty,0]$, with initial condition $\bar Y(0)=Y_0$, decaying
to the completely burned state $(0,\dots,0)$ as $z\to -\infty$,
with $(\bar \rho, \bar u, \bar E)= (\bar \rho, \bar u, \bar E)(\bar Y)$
determined through the generalized Rankine--Hugoniot relations
\begin{equation}\label{gRH}
\begin{aligned}
s\bar \rho- \bar \rho \bar u&= (s \rho- \rho  u)_\pm \\
s\bar \rho \bar u- (\bar \rho \bar u^2+\bar p)&= (s\rho u- (\rho u^2+p))_\pm \\
s\bar \rho \bar {\tilde E}-  (\bar \rho {\bar {\tilde E}}+\bar p)\bar u 
&=(s\rho {\tilde E}-  (\rho {{\tilde E}}+p)u)_\pm
\end{aligned}
\end{equation}
obtained by integrating the remaining traveling-profile equations
\begin{equation}\label{otherprofile}
\begin{aligned}
(s\bar \rho- \bar \rho \bar u)'&=0\\
(s\bar \rho \bar u- (\bar \rho \bar u^2+\bar p))'&=0\\
(s\bar \rho \bar {\tilde E}-  (\bar \rho {\bar {\tilde E}}+\bar p)\bar u)' 
&=0\\
\end{aligned}
\end{equation}
from $0$ to $z$ (where $z<0$) and recalling the Rankine--Hugoniot conditions
\eqref{RH} satisfied across the jump at $z=0$.

That is, strong detonations moving to the right with respect to fluid
velocity $u$ (i.e., $u<s$, where $s$ is the speed of the detonation)
have the structure of an initiating gas-dynamical shock
called the Neumann shock, 
which rapidly compresses the gas, 
raising temperature to the point of ignition, 
followed by a reaction zone (the profile
$(\bar \rho, \bar u, \bar E, \bar Y)$) resolving to the 
final burned state.  
This characteristic ``detonation spike'' 
in temperature and pressure profiles agrees well with observed 
features in laboratory experiments. 

Substituting into \eqref{preY} the first relation in \eqref{gRH}
and introducing the constant 
$ m:= (\rho (s-u))_\pm$, we obtain the simplified
reaction equation
\begin{equation}\label{Yprof}
 Y'= m^{-1}\rho \varphi(T) K Y
\end{equation}
that we will actually use to solve for the profile.  Further simplifying \eqref{gRH}, we obtain
\begin{equation}\label{finalgRH}
\begin{aligned}
s\bar \rho- \bar \rho \bar u&=(s \rho- \rho  u)_\pm \\
s\bar \rho \bar u- (\bar \rho \bar u^2+\bar p)&=(s\rho u- (\rho u^2+p))_\pm \\
s\bar \rho \bar {E}-  (\bar \rho {\bar {E}}+\bar p)\bar u + mq\bar Y &= (s\rho {E}-  (\rho {{E}}+p)u +mqY)_\pm.
\end{aligned}
\end{equation}

An application of the Implicit Function Theorem reveals that \eqref{gRH} 
(as, likewise, the original ODE \eqref{otherprofile})
may be solved for $(\bar \rho, \bar u, \bar E)$ in
terms of $\bar Y$ so long as the gas-dynamical state
$(\bar \rho, \bar u, \bar E)$ remains noncharacteristic with
respect to speed $s$, or, equivalently, the Rankine--Hugoniot relation
\eqref{finalgRH} remains full rank in $(\bar \rho, \bar u, \bar E)$.
For typical reactions and equations of state, in particular ideal gas dynamics
with single exothermic reaction,
this condition holds for all solutions of \eqref{finalgRH}
with $\bar Y_j\ge 0$, except for special limiting values of $s$ for
which the asymptotic state $\bar Y=0$ is characteristic, or ``sonic'';
see, e.g., \cite{LZ1}.
These limiting, characteristic waves are called {\it Chapman--Jouget}
detonations, and have a special place in the theory.
The usual, noncharacteristic type are called {\it overdriven} detonations.

For our present purposes, the main import of characteristicity is that
the eigenvalue equation becomes singular at $x\to -\infty$ in the
coordinates we use here, complicating the discussion.
For simplicity, we restrict hereafter to the overdriven case.
The Chapman--Jouget case may be treated similarly using ideas
of \cite{LS}; see Remark \ref{general}.

\begin{remark}\label{beta}
For the modified Arrhenius ignition function \eqref{Arrhenius}, 
a standard assumption is that $\bar T\ge T^i$, all $x\le 0$,
$T_+\le T_i$, so that $\beta\equiv 1$ for $x\le 0$ and
$\beta\equiv 0$ for $x\ge 0$.  
Under this assumption, the specific form of the
cutoff $\beta$ plays no role in the analysis.
\end{remark}
\bigskip


\section{Linearized stability analysis: the Evans--Lopatinski determinant}\label{s:linearized}

We now carry out a linearized interface analysis, loosely
following \cite{JLW}.\footnote{See also the related \cite{Z1,LZ1,LZ2},
and the original treatments in \cite{Er1,Er2,LS}, etc.} 
Setting $V:=(\rho, u, e)^T$, write \eqref{altznd} in abstract form as
\begin{equation}\label{abznd}
\begin{aligned}
F^0(W)_t + F^1(W)_x=R(W),
\end{aligned}
\end{equation}
$W$, $F^j$, $R\in \mathbb{R}^{3+r}$, where
\begin{equation}\label{abcoefs}
\begin{aligned}
W:=\begin{pmatrix} V\\Y \end{pmatrix}, \quad
F^j:=\begin{pmatrix} f^j(W)\\Y g^j(V) \end{pmatrix},\quad
R:=\begin{pmatrix} QKY\psi(W)\\-KY\psi(W) \end{pmatrix},\\
\end{aligned}
\end{equation}
\begin{equation}\label{abdefs}
\begin{aligned}
f^0&:=\begin{pmatrix} \rho\\\rho u\\ \rho (e+u^2/2) \end{pmatrix},\quad
f^1:=\begin{pmatrix} \rho u\\\rho u^2 + p(\rho,e,Y)\\ 
(\rho (e+u^2/2)+p(\rho,e,Y))u \end{pmatrix},\\
g^0&:=\rho , \quad g^1= \rho u,
\quad
Q:=\begin{pmatrix} 0& \cdots& 0 \\q_1& \cdots & q_r \end{pmatrix}, \quad
\psi:= \rho \phi(T(\rho,e,Y)).\\
\end{aligned}
\end{equation}
with $V$, $f^j\in \mathbb{R}^{3}$, $Y\in \mathbb{R}^r$, 
$g^j$, $\psi \in \mathbb{R}^1$, $Q\in \mathbb{R}^{3\times r}$.

\begin{remark}\label{Yeos}
A minor departure from \cite{Er1,Er2,Z1,JLW} is to admit the possible 
dependence of pressure and temperature on chemical makeup of the gas $(Y)$, 
an important feature in realistic modeling of reactive flow.
\end{remark}

To investigate solutions in the vicinity of a discontinuous 
detonation profile, we postulate existence of a single shock
discontinuity at location $X(t)$, and reduce to a fixed-boundary
problem by the change of variables $x\to x-X(t)$.
In the new coordinates, the problem becomes
\begin{equation}\label{transformed}
\begin{aligned}
F^0(W)_t + (F^1(W) - X'(t) F^0(W))_x=R(W), \quad x\ne 0,
\end{aligned}
\end{equation}
with jump condition
\begin{equation}\label{transformedRH}
\begin{aligned}
X'(t)[F^0(W)] - [F^1(W)]=0, 
\end{aligned}
\end{equation}
$[h(x,t)]:=h(0^+,t)-h(0^-,t)$ as usual denoting jump 
across the discontinuity at $x=0$.

\subsection{Linearization}\label{linearization}
Without loss of generality, suppose for simplicity 
that the background profile $\bar W$ 
is a {\it steady} detonation, i.e., $s=0$, hence 
$(\bar W, \bar X)=(\bar W,0)$ is also a steady solution of 
\eqref{transformed}--\eqref{transformedRH}.
Linearizing \eqref{transformed}--\eqref{transformedRH} about
the solution $(\bar W,0)$, we obtain the {\it linearized equations} 
\begin{equation}\label{lin}
\begin{aligned}
A^0(W_t - X'(t)\bar W'(x)) + (A^1W)_x =CW,
\end{aligned}
\end{equation}
\begin{equation}\label{linRH}
\begin{aligned}
X'(t)[F^0(\bar W)] - [A^1 W]=0, \quad x=0,
\end{aligned}
\end{equation}
\begin{equation}\label{coeffs}
A^j:= (\partial/\partial W)F^j,
\quad
C:= (\partial/\partial W)R.
\end{equation}

\medskip
\subsection{Reduction to homogeneous form}\label{homogeneous}

As pointed out in \cite{JLW}, it is convenient for the stability
analysis to eliminate the front from the interior equation \eqref{lin}.
Therefore, we reverse the original transformation to linear order
by the change of dependent variables
\begin{equation}\label{wsharp}
W\to W- X(t)\bar W'(x),
\end{equation}
following the calculation 
\[
W(x-X(t),t))-W(x,t)\sim X(t)W_x(x,t)
\sim X(t) \bar W'(x).
\]
approximating to linear order the original, nonlinear transformation.
Substituting \eqref{wsharp} in \eqref{lin}--\eqref{linRH}, and
noting that $x$-differentiation of 
the steady profile equation $F^1(\bar W)_x=R(\bar W)$ gives
\begin{equation}\label{cancel}
(A^1(\bar W)\bar W'(x))_x=C(\bar W)\bar W'(x),
\end{equation}
we obtain modified, {\it homogeneous} interior equations
\begin{equation}\label{flin}
\begin{aligned}
A^0W_t + (A^1W)_x =CW
\end{aligned}
\end{equation}
agreeing with those that would be obtained by a naive
calculation without consideration of the front,
together with the modified jump condition
\begin{equation}\label{flinRH}
\begin{aligned}
X'(t)[F^0(\bar W)]-X(t)[A^1 \bar W'(x)] - [A^1 W]=0 
\end{aligned}
\end{equation}
correctly accounting for front dynamics.

The reduction to homogeneous interior equations 
puts the linearized problem in a standard linear
boundary-value-problem format for which stability may
be investigated in straightforward fashion by the construction of an 
{\it Evans/Lopatinski determinant}.
Besides simplifying 
considerably 
Erpenbeck's original derivation
of his equivalent {\it stability function} \cite{Er2}, the
homogeneous format makes possible the application of standard
numerical Evans function techniques for its evaluation.
This useful reduction was first carried out, in slightly
different form, in \cite{JLW}.
The transformation \eqref{wsharp} is of general use in interface
problems, comprising the ``good unknown'' of Alinhac \cite{Al}.
A similar discussion in the simpler context of shock waves
may be found in \cite{Fo}; however, in this case,  $\bar W'(x)\equiv 0$, 
and so the transformation \eqref{wsharp} does not make itself
evident, nor do front dynamics modify \eqref{flinRH}.

\medskip

\subsection{The stability determinant}
Seeking normal mode solutions $W(x,t)=e^{\lambda t}W(x)$,
$X(t)=e^{\lambda t}X$,
$W$ bounded, of the linearized equations \eqref{flin}--\eqref{flinRH},
we are led to the generalized eigenvalue equations
\[
(A^1W)' = (-\lambda A^0  + C)W, \quad x\ne 0,
\]
\[
X(\lambda [F^0(\bar W)]-[A^1 \bar W'(x)]) - [A^1 W]=0, 
\]
where ``$\prime$'' denotes $d/dx$, or, setting $Z:=A^1W$, to
\begin{equation}\label{eig}
\begin{aligned}
Z' = GZ, \quad x\ne 0,
\end{aligned}
\end{equation}
\begin{equation}\label{eigRH}
\begin{aligned}
X(\lambda [F^0(\bar W)]-[A^1 \bar W'(x)]) - [Z]=0, 
\end{aligned}
\end{equation}
with
\begin{equation}\label{G0}
G:=(-\lambda A^0  + C)(A^1)^{-1}.
\end{equation}
\medskip

Here, we are implicitly using the following elementary observation.

\begin{lemma}\label{ainvert}
$A^1(\bar W(x))$ is invertible for all $x$ such that $\partial f/\partial V$
is invertible (i.e. $V$ is noncharacteristic as a gas-dynamical state
with $Y$ held fixed).
\end{lemma}

\begin{proof}
Similarly as in the discussion of existence of steady profiles,
we may by subtracting $Y$ times the first row of $A^1$ 
from the block $Y$-row, reduce $A^1$ to block upper-triangular
form, with diagonal blocks $\partial f/\partial V$ and
$g^1(V,Y)I_{r\times r}$ with $g^1(V,Y)=\rho u \ne 0$.
\end{proof}

\begin{remark}
As discussed in Section \ref{znd}, this assumption is
essentially necessary already for existence of a steady profile.
In particular, it is satisfied for the usual ideal gas equation of state.
%
\end{remark}

We require also the following fundamental properties. 

\begin{lemma}[\cite{Er1,Er2,JLW}] \label{lem:splitting} 
On $\R \lambda >0$, the limiting $(3+r)\times (3+r)$
coefficient matrices $G_\pm:= \lim_{z\to \pm \infty} G(z)$ 
have unstable subspaces of fixed rank: full rank $3+r$ for
$G_+$ and rank $2+r$ for $G_-$.
Moreover, these subspaces have continuous limits as $\R \lambda \to 0$.
\end{lemma}

\begin{proof}
Straightforward calculation using the fact that $G_\pm$
are block upper-triangular in $(V,Y)$;
see, e.g., \cite{Er1,Er2,Z1,JLW} in the case that $f$, $g$
depend only on $V$.
\end{proof}

\begin{corollary}[\cite{Z1,JLW}] \label{cor:splitting} 
On $\R \lambda >0$, the only bounded solution of \eqref{eig}
for $x>0$ is the trivial solution
 $W\equiv 0$.  For $x<0$, the bounded solutions consist of
an $(r+2)$-dimensional
subspace $\Span \{Z_1^+, \dots, Z_{r+2}^+\}(\lambda,x)$
of exponentially decaying solutions, analytic in $\lambda$ and
tangent as $x\to -\infty$ to the subspace of exponentially decaying
solutions of the limiting, constant-coefficient equations
$Z'=G_-Z$; moreover, this subspace has a continuous limit
as $\R \lambda \to 0$.
\end{corollary}

\begin{proof}
The first observation is immediate, using the fact that $G$
is constant for $x>0$.
The second follows from asymptotic ODE theory, 
using the ``gap'' or ``conjugation'' lemmas of \cite{GZ,KS}, \cite{MZ1}
together with the fact that $G$ decays exponentially
to its end state as $x\to -\infty$.
See \cite{JLW,Z3,Z6} for details.
\end{proof}

\begin{definition}\label{evans}
We define the {\it Evans--Lopatinski determinant}
\begin{equation}\label{evanseq}
\begin{aligned}
D(\lambda)&:=
\det \begin{pmatrix}
Z_1^-(\lambda,0), & \cdots, & Z_{r+2}^-(\lambda,0), &
\lambda [F^0(\bar W)]-[A^1 \bar W'(x)]\\
\end{pmatrix}\\
&=
\det \begin{pmatrix}
Z_1^-(\lambda,0), & \cdots, & Z_{r+2}^-(\lambda,0), &
\lambda[F^0(\bar W)]+A^1\bar W'(0^-)\\
\end{pmatrix},\\
\end{aligned}
\end{equation}
where $Z^-_j(\lambda,x)$ are as in Corollary \ref{cor:splitting}.
\end{definition}

The function $D$ is exactly the {\it stability function} 
derived in a different form by Erpenbeck \cite{Er2};
see Section \ref{s:erp} below.
The formulation \eqref{evanseq} is of the standard form arising
in the simpler context of (nonreactive) shock stability \cite{Maj2, Er1}.
Evidently (by \eqref{eigRH} combined with Corollary \ref{cor:splitting}), 
$\lambda$ is a generalized eigenvalue/normal mode
for $\R \lambda \ge 0$ if and only if $D(\lambda)=0$.

\begin{remark}\label{Krmk}
As noted in \cite{Z3,Z6}, consideration of 
the traveling-wave equation $F(W)'=AW'=R(W)$ yields the simpler formula
\begin{equation}\label{simple}
D(\lambda)=
\det \begin{pmatrix}
Z_1^-(\lambda,0), & \cdots, & Z_{r+2}^-(\lambda,0), &
\lambda[F^0(\bar W)]+R(\bar W(0^-))\\
\end{pmatrix}.
\end{equation}
\end{remark}

\subsection{Dual formulation}

The $(n+r)\times (n+r)$ determinant \eqref{simple}
may be expressed more 
succinctly 
in dual form
\begin{equation}\label{dualevans}
\begin{aligned}
D(\lambda)&=
\tilde Z^-(\lambda, 0) \cdot (\lambda[F^0(\bar W)]+R(\bar W(0^-)),
\end{aligned}
\end{equation}
where $\tilde Z^-(\lambda, x)$ is the cross product
$Z_1^-\wedge  \dots \wedge  Z_{r+2}^-(\lambda,x)$ defined by 
\[
\tilde Z^-\cdot x=\det \begin{pmatrix} Z_1^-,& \cdots, & Z_{r+2}^-, & x \end{pmatrix}.
\]
The vector $\tilde Z^-$ may alternatively be characterized
directly as the unique up to constant factor
bounded solution on $x\le 0$ of the adjoint ODE
\begin{equation}\label{dualeig}
\begin{aligned}
\tilde Z' = -G^*\tilde Z,
\end{aligned}
\end{equation}
which, as $x\to -\infty$ is both exponentially decaying and tangent 
to the corresponding exponentially decaying one-dimensional subspace of
bounded solutions of the limiting constant-coefficient equations
$\tilde Z' = -G^*_-\tilde Z$.
It may be specified analytically in $\lambda$ by the additional requirement 
\begin{equation}\label{abinit}
\tilde \Pi (\tilde Z^-)^{conj}(-M)=\ell(\lambda),
\end{equation}
$M>0$, where $\ell$ is an analytically chosen left eigenvector
of $G_-(\lambda)$ associated with the unique eigenvalue $g_-(\lambda)$
of negative real part and $\tilde \Pi$ the associated eigenprojection.
Here, and elsewhere, $^{conj}$ denotes complex conjugate.
By \eqref{abinit} together with the tangency property, 
$\tilde Z^-$ is well-approximated at $x=-M$, 
for $M>0$ sufficiently large, by
\begin{equation}\label{numinit}
\tilde Z^-(-M)=\ell^{conj}(\lambda).
\end{equation}

This reduces 
the approximate
evaluation of $D(\cdot)$ to the straightforward and
extremely well-conditioned numerical problem of integrating a 
single
exponentially growing (in forward direction) mode from $x=-M$ to $x=0$.
The stability of the computation derives from the fact that
errors lying in other, exponentially decaying modes, are
exponentially damped \cite{Z5}.

\medskip
{\bf Alternate initialization.}
Alternatively, following \cite{Br1,Br2,BrZ,BDG},
$\tilde Z^-$ may be specified by boundary conditions at $-\infty$, via
\begin{equation}\label{infabinit}
\lim_{x\to -\infty}
e^{g_-^{conj}x} 
\tilde Z^-(x)=\ell^{conj}(\lambda),
\end{equation}
whence \eqref{numinit} becomes 
\begin{equation}\label{infnuminit}
\tilde Z^-(-M)=
e^{-g_-^{conj}M} \ell^{conj}(\lambda).
\end{equation}
This 
is the method that we 
prescribe here.
It 
has the advantage of removing the dependence 
of $\tilde Z^-$ on the artificial parameter $M$,
allowing the flexible choice of $M$ in different parameter regimes,
as dictated by numerical considerations, while preserving
analyticity.
However, in practice, there is usually not much difference between 
\eqref{numinit} and \eqref{infnuminit}.
In particular, if, as in \cite{LS}, one is not interested in analyticity,
then one may vary $M$ freely in \eqref{numinit} as well.

\section{Relations to other methods}\label{s:relations}
\subsection{ Relation to the method of Lee and Stewart.}
Denoting by $Z_0$ the solution on $x\le 0$ 
of the forward eigenvalue ODE \eqref{eig} with
initial conditions $Z_0(0):= \lambda[F^0(\bar W)]+R( W(0^-))$,
we have by standard duality properties that
\begin{equation}\label{const}
\tilde Z^-\cdot Z_0(\lambda,x)\equiv D(\lambda)
\end{equation}
is independent of $x\le 0$,
or $( \tilde Z^-\cdot Z_0)'(\lambda,x)\equiv 0$.
 Taking $x=-M$ and recalling \eqref{numinit},
we arrive at the alternative Evans--Lopatinski approximation 
\begin{equation}\label{LSapprox}
D(\lambda)\sim  \ell^{conj}(\lambda) \cdot Z_0(\lambda,-M)
\end{equation}
used by Lee and Stewart \cite{LS}, where $\ell^{conj}\cdot Z_0(-M)=0$ 
is their ``nonradiative condition'' enforcing boundedness of $Z_0$.
The solution of $Z_0$ from $x=0$ to $x=-M$, on the other hand,
is numerically 
comparatively
ill-conditioned in the vicinity of roots
of $D(\cdot)$, since $Z_0$ in this regime is approximately 
exponentially decaying in the backward direction while errors
are exponentially growing.\footnote{
More precisely, they solve the inhomogeneous equations
$Z_0'=GZ_0+ \lambda \bar W'(x)$ 
with initial data $\hat Z_0(0):=\lambda [\bar W]$,
and compute 
$\ell^{conj}(\lambda) \cdot Z_0(\lambda,-M) \sim
\ell^{conj}(\lambda) \cdot \hat Z_0(\lambda,-M) $,
which is numerically equivalent.
Here we are using $\hat Z_0- Z_0= \bar W'(x)\to 0$ as $x\to -\infty$.}
The version \eqref{dualevans} is therefore much preferable from
the numerical point of view, at least when used (as here, and in \cite{LS})
as a shooting method.

\subsection{Relation to the method of Erpenbeck}\label{s:erp}
Erpenbeck \cite{Er3} computes $\tilde Z^-$ in much the same
way as we do here. 
However, in place of the homogeneous duality relation \eqref{const},
he uses the ``inhomogeous Abel relation''
\begin{equation}\label{abel}
 (\tilde Z^-\cdot \hat Z_0)'(\lambda,x)) = \tilde Z^\cdot \lambda \bar W'(x),
\end{equation}
valid for the solution $\hat Z_0$ of the inhomogeneous equation
$Z'=GZ+ \lambda \bar W'(x)$ with initial data $\hat Z_0(0):=\lambda [\bar W]$
deriving from the unmodified equations \eqref{lin}--\eqref{linRH},
together with $\bar W'(-\infty)=0$, to evaluate
$$
D(\lambda)= \int_{-\infty}^0 \tilde Z^-(y)\cdot\lambda \bar W'(y) dy
+ \tilde Z^-(0)\cdot \lambda [\bar W].
$$
Though it is mathematically equivalent to the homogeneous scheme described
above, this  has the disadvantage that it is difficult to implement adaptive
control on truncation error simultaneously for the ODE and quadrature
steps.
Indeed, the method is in general a bit more cumbersome to implement
and understand than either of the previous two described methods.
As a one-time cost, the latter is a rather minor point. 
However, the implications of the former for performance appear
to be significant.
Our experience in similar Evans function-type shooting computations 
\cite{BrZ,HuZ2,BHRZ,HLZ} of spectra of asymptotically constant-coefficient
operators
is that a fixed-step scheme can be orders of magnitude slower 
than a comparable adaptive scheme; see \cite{Z5} for a general 
discussion of performance of numerical Evans/Lopatinski solvers.
Moreover, even in the solution of $\tilde Z$ alone, the use
of an adaptive solver without factoring out expected decay is
much less effective in our experience (``numerical pitfall 3'' of \cite{Z5}).

\subsection{Expression as boundary-value solver}
We mention in passing an alternative ``local Evans function'' formulation in the spirit of \cite{LS}, suggested by Sandstede \cite{S2} as a general method for numerical Evans function investigations using {\it collocation/continuation} rather than shooting.  By the analysis of the previous subsections, we may recast the eigenvalue equation \eqref{eig}--\eqref{eigRH}
as in \cite{LS} as an overdetermined two-point boundary-value problem
$Z'=GZ$ with $r+4$ boundary conditions
\begin{equation}\label{bvp}
Z(0):= \lambda[F^0(\bar W)]+R(\bar W(0^-)),
\quad
\lim_{x\to -\infty} \ell^{conj}\cdot Z(x)=0.
\end{equation}
Relaxing at random one of the $r+3$ conditions at $x=0$,
say the requirement on the $j$th coordinate,
we generically obtain a well-posed boundary-value problem
with the correct number $r+3$ of boundary conditions;
one of the coordinates will always suffice.
More, the projective boundary-condition at $x=-\infty$
is numerically ``correct'', making this problem 
extremely well-conditioned for solution by collocation/continuation methods
(see, e.g., \cite{S1}).
Defining $Z(\lambda,x)$ to be the solution of this relaxed problem,
we may then define a local, analytic Evans function
\[
\tilde D(\lambda):= e_j\cdot (Z(\lambda,0)- 
(\lambda[F^0(\bar W)]+R(\bar W(0^-))))
\]
that is numerically well-conditioned and
vanishes if and only if $\lambda$ is an eigenvalue.
This gives a second way to convert \eqref{bvp} into a numerically
well-conditioned problem, though the speed and simplicity
of shooting is lost in this approach, along with global analyticity
useful for winding number calculations.
We shall not investigate this method here, but note that it could
be useful in extreme conditions such as the ultra-high activation
energy limit \cite{BN}.


\section{Numerical implementation}\label{s:algorithm}

We now describe in detail the numerical algorithm
proposed to compute \eqref{dualevans}, following the general
approach set out in \cite{BrZ,HuZ2,Z4,Z5}.

\subsection{Computing the profile}\label{init}

In Evans function computations, a delicate aspect
is often the computation of the background nonlinear profile.
We sidestep this issue by the explicit solution technique used in
\cite{Er1,Er2,LS}, modified slightly to accomodate the multi-species case
(specifically, the simplified uniform ignition one considered here).

Introducing the new variable $y$ defined by
\begin{equation}\label{yeq}
 dy/dx = m^{-1}\rho \varphi(T), \quad y(0)=0,
\end{equation}
where $m:= (\rho (s-u))_\pm$,
we reduce the reaction equation \eqref{Yprof} to
\begin{equation}\label{Yseq}
 dY/dy= K Y, \quad  Y(0)=Y_0,
\end{equation}
obtaining an explicit solution
\begin{equation}\label{Yform}
Y(y)= e^{Ky}Y_0
\end{equation}
from which the full profile can be recovered through \eqref{finalgRH},
either by explicit calculation, as carried out for ideal gas dynamics in
Appendix \ref{B}, or, more generally, by Newton iteration.

\begin{remark}\label{general}
In the single-species case, \eqref{Yseq} reduces to
the change of coordinates $x\to y:=\log Y$ used in \cite{LS};
general, nonlinear rate laws, or Chapman--Jouget waves, may be accomodated
by a change of variables 
$x\to Y^{r}$ for appropriate $r$, as discussed in \cite{LS}.
\end{remark}

\subsection{Computing the stability determinant}\label{stab}

The linearized stability analysis can then be carried out in
the variable $y$ defined in \eqref{yeq}, 
using the instantaneous change of variables formula
\begin{equation}\label{inst}
 dx/dy = m/(\bar \rho \varphi(\bar T)), \quad y\le 0.
\end{equation}

\begin{remark}\label{equiv}
Since the righthand side of \eqref{inst} is uniformly positive and
bounded, the variables $x$ and $y$ are equivalent in the
sense that $Cx\le y\le x/C$ for $x$, $y\le0$, for some $C>0$.
\end{remark}

Specifically, we solve from $y=-M$ to $y=0$ the ODE
$
(d/dy)\tilde Z = 
 -m/(\bar \rho \varphi(\bar T)) G^*(y)\tilde Z,
$
with initial condition 
$\tilde Z^-(-M)= e^{-g_-^{conj}(\lambda) M} \ell^{conj}(\lambda)$,
$M>0$ sufficiently large, where the vector $\ell(\lambda)$ 
and limiting eigenvalue $g_1=(u_-+c_-)^{-1}$
are as computed in eqs. \eqref{genellY} and \eqref{finalellV} of 
Appendix \ref{A}, 
the coefficient $G(\lambda, \bar V, \bar Y)$ 
is as described in 
eqs. \eqref{G0}, \eqref{coeffs}, and \eqref{abcoefs}--\eqref{abdefs},
and the profile $(\bar V, \bar Y)(y)$ 
is as computed in Appendix \ref{B}.
As prescribed in \eqref{dualevans},
we may then compute the stability determinant
$
D(\lambda) = \tilde Z^-(\lambda, 0) \cdot (\lambda[F^0(\bar W)]+
R(\bar W'(0^-))).
$

More precisely, we may solve the numerically more advantageous equations
\begin{equation}\label{2ydualeig}
\begin{aligned}
(d/dy)\hat Z = 
 -(m/\bar \rho \varphi(\bar T)) (G(y)+g_-(\lambda)I)^*\hat Z,
\end{aligned}
\end{equation}
with initial conditions
$\hat Z(y):= e^{
 -(m g_-^{conj} /\bar \rho \varphi(\bar T) y}\tilde Z(y)$,
and compute
\[
D(\lambda) = \hat Z^-(\lambda, 0) \cdot (\lambda[F^0(\bar W)]+R(\bar W(0^-))).
\]
This may readily be computed with good results by 
an adaptive solver such as the standard RK45; see
\cite{BrZ,HuZ2,Z5} for further discussion.

\subsection{Determination of stability: winding number vs. stability curves}\label{winding}

With an Evans solver in hand, stability may be checked either
by winding number computations as in \cite{Er3,BHRZ,HLZ}, or by root-following
methods based on the Implicit Function Theorem, as in \cite{LS}.
In the first method, a large semicircle $S$ centered at the origin and
lying in $\Re \lambda\ge 0$ is mapped by $D$, and the number of zeros
of $D$ (unstable normal modes) lying within $S$
computed using the principle of the argument, making use of the underlying
analyticity of $D$.
Unstable modes lying outside $S$ may be excluded by a separate, asymptotic,
argument based on high-frequency behavior of $D$ \cite{Br1,Br2,HLZ};
for implementations in the context of ZND, see
\cite{Z6,LWZ} (analytical) or \cite{BZ,BZ2} (numerical).
In the second method, individual roots are followed, avoiding the need
to compute around a contour, but typically requiring
an extra Newton iteration with each change in model parameters;
see, for example, \cite{LS,ShS}.
Both are by now completely standard.


\section {A simple model problem}\label{s:simplemodel}

We conclude by an examination of efficiency within
the context of a simple but illustrative model problem.
Consider the ODE
\begin{equation}\label{e:model}
y' = A(x,\lambda) y, \qquad
A(x,\lambda) = \lambda \begin{pmatrix} \frac{1}{2} & 0\\ 
\frac{1}{c}e^{2x} & \frac{-1}{2} \end{pmatrix}
\end{equation}
defined on $-\infty<  x\le 0$, $x\in \mathbb{R}$, $\lambda \in \mathbb{C}$, 
$y\in \mathbb{C}^2$,
with boundary conditions 
$y\sim e^{\lambda x/2}(1,0)^T$ as $x\to -\infty$
 and $y(0)=(1,0)^T$,
modeling a variable-coefficient eigenvalue problem of the form arising in ZND,
where the coefficient $c\ne 0$ encodes 
rapidity of exponential decay.
As for ZND, the coefficient matrix is exponentially asymptotically
constant as $x\to -\infty$, with size growing linearly in $\lambda$,
and has a unique decaying mode as $x\to -\infty$ for all $\Re \lambda>0$,
extending continuously to $\Re \lambda=0$.
Thus, we may expect somewhat similar behavior, 
at least away from the high-activation energy ``square-wave'' regime.

In this context, our proposed algorithm consists of factoring out
the expected decay $e^{\lambda x/2}$ from the solution to obtain
a ``neutral'' equation
\begin{equation}\label{e:modmodel}
\hat y' = \hat A(x,\lambda) \hat y, \qquad
\hat A(x,\lambda) = \lambda \begin{pmatrix} 0 & 0\\ 
\frac{1}{c}e^{2x} & -1\end{pmatrix},
\end{equation}
$y:= \hat ye^{\lambda x/2}$, then solving \eqref{e:modmodel}
from $x=-M$ to $x=0$ and checking whether $\hat y(0)$
lies parallel to $(1,0)^T$.  
For reasonable values of $c$, a computational domain of $M=5$ is sufficient.
The method of Lee-Stewart, consists roughly of integrating the original
equation \eqref{e:model} from $x=0$ to $x=-M$;
the method of Erpenbeck consists roughly of integrating \eqref{e:model}
from $x=-M$ to $x=0$ without first factoring out expected exponential decay.
For comparison, we considered also a worst-case scenario with maximum
amplification of error modes, integrating
\eqref{e:modmodel} from $x=0$ to $x=-\infty$.

We computed all with the adaptive-mesh RK45 algorithm (ode45) supported
in MATLAB,\footnote{
In practice, faster than corresponding fixed-mesh methods \cite{Z5,BZ,BZ2}.}
with error tolerance set at the standard level $10^{-5}$ used
for Evans computations \cite{HLZ,BHZ,BLZ,BLeZ},
measuring efficiency by the number of mesh points/function calls 
required to complete the computation.
Extreme cases are $\lambda$ real- the ``best'' case,
with a spectral gap between exponentially growing and 
exponentially decaying modes at $-\infty$-
and $\lambda $ imaginary- the ``worst'' case from our standpoint, 
with neither spectral gap nor exponential decay.
From the standpoint of the Lee-Stewart method, the best and worst cases
would appear to be reversed.

The results, displayed in Tables 1 and 2 for a typical value $c=10$,
indicate that the proposed new algorithm performs $1$-$5$ times faster
than (adaptive versions of) either the Erpenbeck or Lee-Stewart methods,
depending on the value of $\lambda$, with particular improvement
as $|\lambda|$ becomes large.
%
%
%
It should be noted, moreover, that this is only a comparison of {\it speed}
(number of mesh points) for the various methods to produce output with
fixed truncation error.
If we consider also {\it accuracy}, i.e., convergence
error, then the results could be expected to be more
dramatic, since both Lee-Stewart and Erpenbeck methods
are numerically less well-posed than the forward ``neutral''
algorithm that we propose.


%
%

\begin{table}
\begin{center}
\begin{tabular}{| r | c | c | c | c | c | c |}\hline
& \multicolumn{6}{c|}{mesh points} \\
\hline
& \multicolumn{3}{c|}{forward integration} &  \multicolumn{3}{c|}{backward
integration}\\
\hline
$\lambda$ & $c=10$ & 100 & 1000 & $c=10$ & 100 & 1000\\
\hline
1.0+ 0i& 19 & 14 & 12 & 26 & 24 & 19 \\
4.0+ 0i& 43 & 29 & 19 & 94 & 92 & 88 \\
16.0+ 0i& 107 & 76 & 51 & 363 & 361 & 357 \\
64.0+ 0i& 261 & 191 & 138 & 1438 & 1436 & 1432 \\
256.0+ 0i& 657 & 519 & 427 & 3177 & 3186 & 3192 \\
\hline
0.4+ 0i& 14 & 12 & 11 & 17 & 14 & 11 \\
0.4+ 1i& 17 & 13 & 12 & 30 & 27 & 18 \\
0.4+ 4i& 43 & 29 & 19 & 100 & 97 & 73 \\
0.4+16i& 111 & 77 & 51 & 385 & 382 & 296 \\
0.4+64i& 317 & 224 & 177 & 1528 & 1523 & 1185 \\
0.4+256i& 1088 & 870 & 827 & 6104 & 6086 & 4738 \\
\hline
\end{tabular}
\caption{Runs for Eq. \eqref{e:modmodel}.
Forward corresponds to our proposed method, with expected decay
factored out.  Backward is a worst-case scenario not corresponding
to any of the methods considered.}
\label{comparison}
\end{center}
\end{table}

%
%

\begin{table}
\begin{center}
\begin{tabular}{| r | c | c | c | c | c | c |}\hline
& \multicolumn{6}{c|}{mesh points} \\
\hline
& \multicolumn{3}{c|}{forward integration} &  \multicolumn{3}{c|}{backward
integration}\\
\hline
$\lambda$ & $c=10$ & 100 & 1000 & $c=10$ & 100 & 1000\\
1.0+ 0i& 23 & 19 & 15 & 19 & 17 & 15 \\
4.0+ 0i& 61 & 58 & 56 & 52 & 50 & 49 \\
16.0+ 0i& 181 & 181 & 181 & 186 & 184 & 183 \\
64.0+ 0i& 719 & 719 & 719 & 723 & 721 & 721 \\
256.0+ 0i& 2868 & 2868 & 2868 & 2873 & 2871 & 2870 \\
\hline
0.4+ 0i& 16 & 13 & 12 & 17 & 13 & 12 \\
0.4+ 1i& 20 & 17 & 15 & 20 & 17 & 15 \\
0.4+ 4i& 55 & 52 & 50 & 54 & 52 & 50 \\
0.4+16i& 196 & 194 & 193 & 197 & 195 & 193 \\
0.4+64i& 765 & 765 & 765 & 775 & 771 & 765 \\
0.4+256i& 3055 & 3055 & 3055 & 3084 & 3074 & 3055 \\
\hline
\hline
\end{tabular}
\caption{Runs for Eq. \eqref{e:model}.
Forward corresponds to Erpenbeck method, backward to Lee--Stewart method.}
\label{comparison2}
\end{center}
\end{table}


\appendix
\section{Calculation of $\ell$}\label{A}
In this appendix, we show how to calculate for general equations of state
the initializing vector $\ell(\lambda)$
used in \eqref{numinit}, the unique stable left
eigenvector of the limiting coefficient matrix 
\begin{equation}\label{Gminus}
\begin{aligned}
G_-&= (-\lambda A^0_- +C_-) (A^1_-)^{-1}
=
\begin{pmatrix}
-\lambda f^0_{V-} (f^1_{V-})^{-1} & 
(\lambda f^0_{V-}(f^1_{V-})^{-1}f^1_{Y-} + QK\psi_-)(g^1_-)^{-1}
\\
 \\
0 & 
(-\lambda g^0_- -K\psi_-) (g^1_-)^{-1}
\end{pmatrix},
\end{aligned}
\end{equation}
where for a general function $h(V,Y)$, we use $h_-$ to denote $h(V_-,Y_-)$.  Here, we have
strongly used $Y_-=0$ to obtain the simple upper block-triangular form.

By the upper block-triangular form of $G_-$, and the fact that the lower
right-hand block has spectrum of positive real part for $\R \lambda >0$
(since $g^0>0$ always, $g^1<0$ for right-moving detonations, and
$-K$ is assumed to have spectrum of negative real part),
we find that $\ell^T$ must be of form $(\ell_V^T, \ell_Y^T)$, where
$\ell_V$ is the unique unstable eigenvector, associated with
eigenvalue $\alpha$, of the purely gas-dynamical matrix
$ f^0_{V-} (f^1_{V-})^{-1} $, and 
\begin{equation}
\label{genellY}
\ell_Y^T= \ell_V^T (\lambda f^0_{V-}(f^1_{V-})^{-1}f^1_{Y-} + QK\psi_-) (\lambda( g^0_-- \alpha g^1_-)I +K\psi_-)^{-1}.
\end{equation}

To determine $\alpha$, $\ell_V$, and thereby $\ell_Y$,
we observe that $ f^0_{V-} (f^1_{V-})^{-1} $, 
is related by similarity tranform $M\to (f^0_{V-})^{-1}M f^0_{V-}$
to the inverse $(f^1_{V-})^{-1} f^0_{{V-}}$ of the hyperbolic
convection matrix $(f^0_{V})^{-1}f^1_V$
of the nonreactive Euler equations 
$
V_t+ (f^0_{V})^{-1}f^1_VV_x=0
$
written in nonconservative form in $V$ coordinates with $Y\equiv 0$.
Thus, $\alpha^{-1}$ is an eigenvalue of $(f^0_{V})^{-1}f^1_V$,
i.e., a hyperbolic characteristic speed of the non-reactive Euler equations,
and $\ell_V^T= \tilde \ell^T (f^0_{V-})^{-1}$, where $\tilde\ell$
is the associated left characteristic direction (eigenmode).

Noting that $\alpha^{-1}$, as the unique positive characteristic
at state $V=V_-$, must be the largest characteristic speed,
we have by standard formulae \cite{Sm,Se1,Se2,LZ1,Z2} or direct calculation
\begin{equation}\label{finalellV}
\ell_V^T= \tilde \ell^T (f^0_{V-})^{-1}=
( p_\rho-cu+ \rho^{-1}p_e(u^2/2-e), c- \rho^{-1} p_eu, \rho^{-1} p_e),
\end{equation}
determining $\ell^T(\lambda)=(\ell_V^T,\ell_Y^T)(\lambda)$
through \eqref{genellY}.
Note that $\ell_V$ is independent of $\lambda$.
For $Y$-independent equations of state, \eqref{genellY}
simplifies considerably, to
$
\ell_Y^T= \ell_V^T QK\psi_- (\lambda( g^0_-- \alpha g^1_-)I +K\psi_-)^{-1}.
$

\begin{remark}\label{quickell}
Noting that the $e$-component $\rho^{-1}p_e$ of $\ell_V$ does not
vanish in \eqref{finalellV}, we may alternatively rescale
by $\rho/p_e$ to obtain an analytic choice of form
$\ell^T=(*,*,1,*)$ convenient for numerical solution.
\end{remark}


\subsection{Alternative, numerical computation}
Alternatively, an analytic choice of $\ell$ may be determined
numerically by solution of Kato's ODE \cite{K}
as described in \cite{BrZ,HuZ2,Z4,Z5}.  
For $\Re \lambda$ bounded from zero, this
involves finding numerically at each $\lambda$-value
the unique stable left and right eigenvectors of $G_-$
and computing the associated eigenprojection for use in
the Kato ODE as in the general problem-independent
method of \cite{BrZ,HuZ2,Z4,Z5}.
At or near $\Re \lambda =0$, however,
this method must be modified, since the stable eigenvector becomes
neutral at $\Re \lambda=0$.
A simple resolution is to notice that, there, the eigenvalues of
$G_-$ consist of a single eigenvalue with strictly positive real
part, which may be discarded, and three eigenvalues
of form $g_j=\alpha_j \lambda$, where $\alpha_j$ (see above) 
are hyperbolic characteristic speeds for the non-reactive Euler equations,
of which the one for which $g_j/\lambda=\alpha_j<0$ is the one 
associated with $\ell$.




\section{Ideal gas profile}\label{B}
In this appendix, we explicitly solve \eqref{finalgRH} for the case of an ideal gas.  Restricting to a steady shock, $s=0$, and using the ideal gas law \eqref{ideal}, we may rewrite \eqref{finalgRH} as 
\begin{equation}
\label{solvegRH}
\begin{aligned}
\bar{\rho}\bar{u}&= \rho_\pm  u_\pm := -m\\
\bar{u} + \Gamma \frac{\bar{e}}{\bar{u}}&= u_\pm + \Gamma \frac{e_\pm}{u_\pm} := b \\
\frac{\bar{u}^2}{2} + (\Gamma+1)\bar{e} + q \bar{Y} &= \frac{u_\pm^2}{2} + (\Gamma+1) e_\pm + q Y_\pm := c.
\end{aligned}
\end{equation}
Combining the second two equations and simplifying gives
$
(\Gamma+2) \bar{u}^2 - 2 (\Gamma+1) b \bar{u} + 2 \Gamma (c - q \bar{Y}) = 0.
$
Solving using the quadratic formula, we obtain
\begin{equation}
\label{solved}
\begin{aligned}
\bar{\rho} &= -\frac{m}{\bar{u}},\quad
\bar{e}= \frac{b \bar{u} - \bar{u}^2}{\Gamma}, \quad
\bar{u}= \frac{\Gamma+1}{\Gamma+2}b \pm \sqrt{\left(  \frac{\Gamma+1}{\Gamma+2} \right)^2 b^2 + \frac{2 \Gamma (q \bar{Y} -c) }{\Gamma+2} },
\end{aligned}
\end{equation}
where we have chosen the negative solution branch for $\bar u$ in
accordance with the fact that $[u]>0$, or, equivalently, $[\rho]<0$,
for a right-moving gas-dynamical shock, so that $\bar u(0^-)<u_+$.
(Recall that $\bar u(0^-)$ and $u_+$ are
the two branches of the square root for $Y=1$, corresponding to the
solutions of the Rankine--Hugoniot conditions for a nonreacting
gas-dynamical shock.)
With \eqref{Yform}, \eqref{solved} gives an explicit expression
for the profile as a function of variable $y$.


For a given Neumann shock, there is a one-parameter
family of possible endstates $(\rho, u, e)_-$ determined by
the value of $q$,
the maximum value of $q$ corresponding to a Chapman--Jouget wave,
for which the argument of the square root vanishes for $y=0$.


\begin{thebibliography}{GMWZ5}

\bibitem{AGJ}
J.~Alexander, R.~Gardner and C.K.R.T.~Jones.
\newblock A topological invariant arising in the analysis of traveling waves.
\newblock {\em J. Reine Angew. Math.} \textbf{410} (1990) 167--212.

\bibitem {AT} G.~Abouseif and T.Y.~Toong,
\newblock Theory of unstable one-dimensional detonations,
\newblock {\em Combust. Flame} 45 (1982) 67--94.

\bibitem{Al}
S.~Alinhac.
\newblock Existence d'ondes de rar\'efaction pour des syst\`emes
  quasi-lin\'eaires hyperboliques multidimensionnels.
\newblock {\em Comm. Partial Differential Equations}, 14(2):173--230, 1989.

\bibitem {BHRZ} B. Barker, J. Humpherys, , K. Rudd, and K. Zumbrun,
\newblock Stability of viscous shocks in isentropic gas dynamics,
\newblock {\em Comm. Math. Phys.}  281  (2008),  no. 1, 231--249.

\bibitem{BHZ}
B.~Barker, J.~Humpherys, and K.~Zumbrun.
\newblock Stability of isentropic parallel mhd shock layers.
\newblock {\em J. Differential Equations}, 249(9):2175--2213, 2010.

\bibitem{STABLAB}
B.~Barker, J.~Humpherys, and K.~Zumbrun.
\newblock STABLAB: A {MATLAB}-based numerical library for {E}vans function computation.
\newblock Available at: http://impact.byu.edu/stablab/.


\bibitem{BLZ}
B.~Barker, O.~Lafitte, and K.~Zumbrun.
\newblock Existence and stability of viscous shock profiles for 2-{D}
  isentropic {MHD} with infinite electrical resistivity.
\newblock {\em Acta Math. Sci. Ser. B Engl. Ed.}, 30(2):447--498, 2010.

\bibitem{BLeZ}
B.~Barker, M.~Lewicka, and K.~Zumbrun.
\newblock Existence and stability of viscoelastic shock profiles, 2010.
\newblock To appear, Arch. Ration. Mech. Anal.

\bibitem{BSYZ} B. Barker, S. Shaw, S. Yarahmadian, and K. Zumbrun, 
\newblock Existence and stability of steady states
of a reaction convection diffusion equation modeling
microtubule formation, 
\newblock to appear, J. Math. Biology.

\bibitem{BZ} B. Barker and K. Zumbrun, 
\newblock Numerical stability of ZND detonations for Majda's model,
\newblock in preparation (2010). 

\bibitem{BZ2} B. Barker and K. Zumbrun, 
\newblock Numerical stability of ZND detonations,
\newblock in preparation.


\bibitem{BMR} A.~Bourlioux, A.~Majda, and V.~Roytburd,
\newblock Theoretical and numerical structure for unstable one-dimensional
detonations.
\newblock {\em  SIAM J. Appl. Math.} 51 (1991) 303--343.

\bibitem{BDG}
T.~J. Bridges, G.~Derks, and G.~Gottwald.
\newblock Stability and instability of solitary waves of the fifth-order
  {K}d{V} equation: a numerical framework.
\newblock {\em Phys. D}, 172(1-4):190--216, 2002.

\bibitem{Br1}
L.~Q. Brin.
\newblock {\em Numerical testing of the stability of viscous shock waves.}
\newblock PhD thesis, Indiana University, Bloomington, 1998.

\bibitem{Br2}
L.~Q. Brin.
\newblock Numerical testing of the stability of viscous shock waves.
\newblock {\em Math. Comp.}, 70(235):1071--1088, 2001.

\bibitem{BrZ}
L.~Q. Brin and K.~Zumbrun.
\newblock Analytically varying eigenvectors and the stability of viscous shock
  waves.
\newblock {\em Mat. Contemp.}, 22:19--32, 2002.
\newblock Seventh Workshop on Partial Differential Equations, Part I (Rio de
  Janeiro, 2001).

\bibitem{BN} J.~Buckmaster and J.~Neves,
\newblock One-dimensional detonation stability: the spectrum for infinite 
activation energy.
\newblock {\em Phys. Fluids 31} (1988) no. 12, 3572--3576.

\bibitem{CHNZ} N.~Costanzino, J.~Humpherys, T.~Nguyen, and K.~Zumbrun,
\newblock Spectral stability of noncharacteristic isentropic Navier-Stokes 
boundary layers. 
\newblock{\em Arch. Ration. Mech. Anal.}  192  (2009),  no. 3, 537--587.

\bibitem{Er1}
J.~J. Erpenbeck.
\newblock Stability of step shocks.
\newblock {\em Phys. Fluids}, 5:1181--1187, 1962.

\bibitem{Er2}
J.~J. Erpenbeck.
\newblock Stability of steady-state equilibrium detonations.
\newblock {\em Physics of Fluids}, 5(5):604--614, 1962.

\bibitem{Er3}
J.~J. Erpenbeck.
\newblock Stability of idealized one-reaction detonations,
\newblock {\em Phys. Fluids}, 7 (1964).


\bibitem{FD}
W.~Fickett and W.~C. Davis.
\newblock {\em Detonation: Theory and Experiment}.
\newblock Dover Publications, 2000.

\bibitem {FW} Fickett and Wood,
\newblock Flow calculations for pulsating one-dimensional
detonations.
\newblock {\em Phys. Fluids} 9 (1966) 903--916.


\bibitem{Fo}
G.~R. Fowles.
\newblock On the evolutionary condition for stationary plane waves in inert and
  reactive substances.
\newblock In {\em Shock induced transitions and phase structures in general
  media}, volume~52 of {\em IMA Vol. Math. Appl.}, pages 93--110. Springer, New
  York, 1993.

\bibitem{GZ}
R.~A. Gardner and K.~Zumbrun.
\newblock The gap lemma and geometric criteria for instability of viscous shock
  profiles.
\newblock {\em Comm. Pure Appl. Math.}, 51(7):797--855, 1998.

\bibitem{HLZ}
J.~Humpherys, O.~Lafitte, and K.~Zumbrun.
\newblock Stability of isentropic {N}avier-{S}tokes shocks in the high-{M}ach
  number limit.
\newblock {\em Comm. Math. Phys.}, 293(1):1--36, 2010.

\bibitem{HLyZ}
J.~Humpherys, G.~Lyng, and K.~Zumbrun.
\newblock Spectral stability of ideal-gas shock layers.
\newblock {\em Arch. Ration. Mech. Anal.}, 194(3):1029--1079, 2009.

\bibitem{HuZ2}
J.~Humpherys and K.~Zumbrun.
\newblock An efficient shooting algorithm for {E}vans function calculations in
  large systems.
\newblock {\em Phys. D}, 220(2):116--126, 2006.

\bibitem{JLW}
H.~K. Jenssen, G.~Lyng, and M.~Williams.
\newblock Equivalence of low-frequency stability conditions for
  multidimensional detonations in three models of combustion.
\newblock {\em Indiana Univ. Math. J.}, 54(1):1--64, 2005.

\bibitem{KS}
T.~Kapitula and B.~Sandstede.
\newblock Stability of bright solitary-wave solutions to perturbed nonlinear
  {S}chr\"odinger equations.
\newblock {\em Phys. D}, 124(1-3):58--103, 1998.

\bibitem{K} T. Kato,
\newblock Perturbation theory for linear operators.
\newblock Springer--Verlag, Berlin Heidelberg (1985).

\bibitem{KaS} A.R.~Kasimov and D.S.~Stewart,
\newblock Spinning instability of gaseous detonations.
\newblock {\em J. Fluid Mech.} 466 (2002), 179--203.

\bibitem{LWZ}
O.~Lafitte, M.~Williams, and K.~Zumbrun.
\newblock 
High-frequency asymptotics and multi-d instability of ZND detonations.
\newblock In preparation.

\bibitem{LS}
H.~I. Lee and D.~S. Stewart.
\newblock Calculation of linear detonation instability: one-dimensional
  instability of plane detonation.
\newblock {\em J. Fluid Mech.}, 216:103--132, 1990.

\bibitem{LZ1}
G.~Lyng and K.~Zumbrun.
\newblock One-dimensional stability of viscous strong detonation waves.
\newblock {\em Arch. Ration. Mech. Anal.}, 173(2):213--277, 2004.

\bibitem{LZ2}
G.~Lyng and K.~Zumbrun.
\newblock A stability index for detonation waves in {M}ajda's model for
  reacting flow.
\newblock {\em Phys. D}, 194(1-2):1--29, 2004.

\bibitem{Maj2}
A.~Majda.
\newblock The stability of multidimensional shock fronts.
\newblock {\em Mem. Amer. Math. Soc.}, 41(275):iv+95, 1983.

\bibitem{MZ1}
C.~Mascia and K.~Zumbrun.
\newblock Stability of large-amplitude viscous shock profiles of
  hyperbolic-parabolic systems.
\newblock {\em Arch. Ration. Mech. Anal.}, 172(1):93--131, 2004.

\bibitem{PW}
R.~L. Pego and M.~I. Weinstein.
\newblock Eigenvalues, and instabilities of solitary waves.
\newblock {\em Philos. Trans. Roy. Soc. London Ser. A}, 340(1656):47--94, 1992.

\bibitem{S1}
B.~Sandstede.
\newblock Stability of travelling waves.
\newblock In {\em Handbook of dynamical systems, Vol. 2}, pages 983--1055.
  North-Holland, Amsterdam, 2002.

\bibitem{S2}
B.~Sandstede.
\newblock private communication, 1999.

\bibitem{Se1}
D.~Serre.
\newblock {\em Systems of conservation laws. 1}.
\newblock Cambridge University Press, Cambridge, 1999.
\newblock Hyperbolicity, entropies, shock waves, Translated from the 1996
  French original by I. N. Sneddon.

\bibitem{Se2}
D.~Serre.
\newblock {\em Systems of conservation laws. 2}.
\newblock Cambridge University Press, Cambridge, 2000.
\newblock Geometric structures, oscillations, and initial-boundary value
  problems, Translated from the 1996 French original by I. N. Sneddon.

\bibitem{ShS}
M.~Short and D.~S. Stewart.
\newblock The multi-dimensional stability of weak-heat-release detonations.
\newblock {\em J. Fluid Mech.}, 382:109--135, 1999.

\bibitem{SK} 
D.~S. Stewart and A.~R. Kasimov,
\newblock On the State of Detonation Stability Theory and Its Application 
to Propulsion,
\newblock {\em Journal of Propulsion and Power}, 22:6, 1230-1244, 2006.

\bibitem{Sm}
J.~Smoller.
\newblock {\em Shock waves and reaction-diffusion equations}.
\newblock Springer-Verlag, New York, second edition, 1994.



\bibitem{TZ3}
B.~Texier and K.~Zumbrun.
\newblock Galloping instability of viscous shock waves.
\newblock {\em Phys. D}, 237(10-12):1553--1601, 2008.

\bibitem{TZ4}
B.~Texier and K.~Zumbrun.
\newblock Hopf bifurcation of viscous shock waves in compressible gas dynamics
  and {MHD}.
\newblock {\em Arch. Ration. Mech. Anal.}, 190(1):107--140, 2008.

\bibitem{TZ5}
B.~Texier and K.~Zumbrun.
\newblock Transition to instability of viscous detonation waves is generically associated with Hopf bifurcation to time-periodic galloping solutions,
\newblock to appear, {\em Comm. Math. Physics} 


\bibitem{Z1}
K.~Zumbrun.
\newblock Multidimensional stability of planar viscous shock waves.
\newblock In {\em Advances in the theory of shock waves}, volume~47 of {\em
  Progr. Nonlinear Differential Equations Appl.}, pages 307--516. Birkh\"auser
  Boston, Boston, MA, 2001.

\bibitem{Z2}
K.~Zumbrun.
\newblock Stability of large-amplitude shock waves of compressible
  {N}avier-{S}tokes equations.
\newblock In {\em Handbook of mathematical fluid dynamics. Vol. III}, pages
  311--533. North-Holland, Amsterdam, 2004.
\newblock With an appendix by Helge Kristian Jenssen and Gregory Lyng.

\bibitem{Z3}
K.~Zumbrun.
\newblock Stability of viscous detonations in the ZND limit.
\newblock to appear, Arch. Rational Mech. Anal.

\bibitem{Z4}
K.~Zumbrun.
\newblock A local greedy algorithm and higher order extensions for global
  numerical continuation of analytically varying subspaces.
\newblock to appear, Quart. Appl. Math.

\bibitem{Z5}
K.~Zumbrun.
\newblock Numerical error analysis for evans function computations: a numerical
  gap lemma, centered-coordinate methods, and the unreasonable effectiveness of
  continuous orthogonalization.
\newblock preprint, 2009.

\bibitem{Z6}
K.~Zumbrun.
\newblock 
High-frequency asymptotics and stability of ZND detonations in
the high-overdrive and small-heat release limits.
\newblock preprint, 2010.

\end{thebibliography}

\end{document}